\date{March 11, 2019}

\documentclass[12pt]{amsart}
\usepackage{latexsym,amsmath,amsfonts,amscd,amssymb}
\usepackage{graphics}
\usepackage{xcolor}
\newtheorem{theorem}{Theorem}[section]
\newtheorem{lemma}[theorem]{Lemma}
\newtheorem{proposition}[theorem]{Proposition}
\newtheorem{corollary}[theorem]{Corollary}

\newtheorem{definition}[theorem]{Definition}

\theoremstyle{remark}

\newcommand{\dd}{\mathrm{d}}

\newcommand{\Res}{\text{Res}}

\newcommand{\Div}{\operatorname{Div}}

\newcommand{\cO}{{\mathcal O}}

\newcommand{\CC}{{\mathbb C}}

\newcommand{\NN}{{\mathbb N}}

\newcommand{\RR}{{\mathbb R}}

\newcommand{\ZZ}{{\mathbb Z}}

\renewcommand{\Re}{\operatorname{Re}}
\renewcommand{\Im}{\operatorname{Im}}

\newcommand{\G}{\Gamma}

\title{On the definition of Euler Gamma function}


\author[R. P\'{e}rez-Marco]{Ricardo P\'{e}rez-Marco}

\address{Ricardo P\'erez-Marco\newline 
\indent  Institut de Math\'ematiques de Jussieu-Paris Rive Gauche, \newline
\indent CNRS, UMR 7586, \newline
\indent  Universit\'e de Paris, B\^at. Sophie Germain, \newline 
\indent 75205 Paris, France}

\email{ricardo.perez-marco@imj-prg.fr}

\begin{document}

\begin{abstract}
We present a new definition of Euler Gamma function. From the complex analysis and transalgebraic viewpoint, 
it is a natural characterization in the space of finite order meromorphic functions. 
We show how the classical theory and formulas develop naturally, and we discuss the relation 
with other definitions. We show in a companion article, that this definition generalizes 
to higher gamma functions and provides an unifying framework for their definition which is more natural 
than the usual Bohr-Mollerup or Lerch approaches.
\end{abstract}

\maketitle

\section{Introduction}

Our first result is a new characterization and  definition of Euler Gamma function. We denote the right half 
complex plane by $\CC_+=\{s\in \CC ; \Re s >0\}$.

\begin{theorem}\label{thm:main-Euler-gamma}
There is one and only one finite order meromorphic function $\Gamma(s)$, $s\in \CC$, 
without zeros nor poles in $\CC_+$, such that $\Gamma(1)=1$, 
$\Gamma'(1)\in \RR$,  and which satisfies the functional equation
 $$
 \Gamma(s+1)=s\, \Gamma(s)
 $$
\end{theorem}

\begin{definition}[Euler Gamma function]
The only solution to the above conditions is the Euler Gamma function.
\end{definition}

\begin{proof}

In the proof we use the elementary theory of entire function and Weierstrass factorization (in the Appendix we remind 
this basic facts).
We first prove the existence and then uniqueness.

\smallskip

\textbf{Existence:}
If we have a  function satisfying the previous conditions then its divisor is contained in $\CC-\CC_+$,
and the functional equation implies that it has no zeros and only has simple poles at the non-positive integers. 
We can construct directly such a meromorphic function $g$
with that divisor, for instance
\begin{equation}\label{eq:W-solution}
g(s)=s^{-1}\prod_{n=1}^{+\infty} \left ( 1+\frac{s}{n}\right )^{-1}e^{s/n} 
\end{equation}
which converges since $\sum_{n\geq 1} n^{-2} <+\infty$, and is of finite order by the classical Weierstrass factorization.
Now, we have that the meromorphic function $\frac{g(s+1)}{s g(s)}$ has no zeros nor poles and it is of finite 
order (as ratio of finite order functions), hence there exists a polynomial $P$ such that 
$$
\frac{g(s+1)}{s g(s)}=e^{P(s)} \ .
$$
Consider a polynomial $Q$ such that 
\begin{equation}\label{eq:W-solution2}
\Delta Q(s) =Q(s+1)-Q(s) =P(s)
\end{equation}
The polynomial $Q$ is uniquely determined from $P$ up to a constant, hence we can choose $Q$ 
such that $e^{Q(0)}=g(1)^{-1}$. Now we have that  $\Gamma (s) = e^{-Q(s)}g(s)$ satisfies the 
functional equation and all the conditions.

\smallskip

\textbf{Uniqueness:}
Consider a second solution $f$. Let $F(s) = \Gamma(s)/f(s)$. Then $F$ is an entire function of finite order 
without zeros, hence we can write $F(s)=\exp{A(s)}$
for some polynomial $A$. Moreover, the functional equation shows that $F$ is $\ZZ$-periodic. Hence, there exists 
an integer $k\in \ZZ$,  such that for any $s\in \CC$,
$$
A(s+1)=A(s)+2\pi i k \ .
$$
It follows that $A(s)=a+2\pi i k s$ for some $a\in \CC$. Since $F(1)=1$, we have $e^a=1$. Since $F'(1) \in \RR$, and 
$F'(1)=F'(1)/F(1)=2 \pi i k  \in \RR$ we have $k=0$, thus $F$ is constant, $F\equiv 1$ and $f=\Gamma$.
\end{proof}

\textbf{Remarks.}

\begin{itemize}
 \item Using the functional equation we can weaken the conditions 
 and request only that the function is meromorphic only on $\CC_+$ with the corresponding 
 finite order growth. We will eventually use that version of the characterization. We can also 
 assume that it is only defined on a cone containing the positive real axes, a  
 vertical strip of width larger than $1$, or in general with any region $\Omega$ which is a 
transitive region for the integer translations and $f$ satisfies the finite order growth condition in 
$\Omega$ when $s\to +\infty$. 
\begin{proposition}\label{prop:domain}
Let $\Omega\subset \CC$ a domain such that for any $s\in \CC$ there exists an integer $n(s)\in \ZZ$ 
such that $s+n(s)\in \Omega$, and $|n(s)|\leq C |s|^d$, for some constants $C, d>0$ depending 
only on $\Omega$. Then any function $\tilde \Gamma$ satisfying a finite order estimate in $\Omega$ and the functional 
equation $\tilde \Gamma(s+1)=s\tilde \Gamma(s)$ when $s, s+1\in \Omega$, 
extends to a finite order meromorphic function on $\CC$.
\end{proposition}

\begin{proof}
Let $\tilde \Gamma$ be such a function. Let $\Omega$ be corresponding region.
Iterating the functional equation we get that $\tilde \Gamma$ extends meromorphically to the whole complex plane. 
Then, if $g$ is the Weierstrass product (\ref{eq:W-solution}) and $Q$ a polynomial given by (\ref{eq:W-solution2}), 
the function $h(s)=\tilde \Gamma(s)/(e^{-Q(s)}g(s))$ 
is a $\ZZ$-periodic entire function. Since $1/(e^{-Q}g)$ is an entire function of finite order, we have in $\Omega$ 
the finite order estimate for $h$. Using that $|n(s)|\leq C|s|^d$, we get that $h$ is of 
finite order, hence $\tilde \Gamma$ is meromorphic 
in the plane of finite order.
\end{proof}

\item Assuming $\Gamma$ real-analytic we get $\Gamma'(1)\in \RR$, but this last  
condition is much weaker. Also, as it follows from the proof, we can 
replace this condition by $\Gamma(a)\in  \RR$ for some $a\in \RR-\ZZ$, or only request that $\Gamma$ is asymptotically real, 
$\lim_{x\in \RR, x\to +\infty} \Im \Gamma (x)=0$. 
Without the condition $\Gamma'(1)\in \RR$
the proof shows that $\Gamma$ is uniquely determined up to a factor $e^{2\pi i k s}$. More precisely, we have

\begin{theorem}\label{thm:main-Euler-gamma_bis}
Let $f$ be a finite order meromorphic function in $\CC$, 
without zeros nor poles in $\CC_+$, such that $f(1)=1$, and satisfying the functional equation
 $$
 f(s+1)=s\, f(s)  \ ,
 $$
 then there exists $k\in \ZZ$ such that 
 $$
 f(s)=e^{2\pi i k s} \Gamma (s) \ .
 $$
\end{theorem}

\medskip

\item We can be more precise in the proof observing that the canonical product (\ref{eq:W-solution}) is of order $1$. 
Hence $g(s+1)/(sg(s))$ is of order $\leq 1$, so the polynomial $P$ has degree at most $1$, and $Q$ has degree at 
most $2$. We can compute these polynomials. Let $\gamma$ be the Euler-Mascheroni constant
$$
\gamma =\lim_{N\to +\infty} \sum_{n=1}^{N} \frac1n -\log N \ .
$$
We have by direct computation $g(1)=e^{\gamma}$. Also we have 
\begin{align*}
\frac{g(s+1)}{sg(s)}&=\lim_{N\to +\infty} \frac{1}{s+1} \prod_{n=1}^N  
\frac{1+\frac{s}{n}}{1+\frac{s+1}{n}} \, e^{1/n}\\
&=\lim_{N\to +\infty} \exp \left ( \sum_{n=1}^N \frac1n -\log (N+1+s)\right ) \\
&=e^\gamma \\
\end{align*}
so $P(s)=\gamma$ is a constant polynomial, and then $Q(s)=\gamma s$. We get:

\begin{corollary}\label{cor:W-product}
 We have
 $$
 \Gamma(s) =s^{-1}e^{-\gamma s}\prod_{n=1}^{+\infty} \left ( 1+\frac{s}{n}\right )^{-1}e^{s/n} \ .
 $$
\end{corollary}

\end{itemize}

 \medskip

 This characterization of the Gamma function came out naturally from related work \cite{MPM1} in the context of higher Barnes 
Gamma functions. We  could not find this characterization in the literature, even for Euler Gamma function. 
After revising the extensive literature on Gamma function and Eulerian integrals (see \cite{PM1}), this definition seems to be new. 
What is closer is Wielandt's beautiful characterization of Euler Gamma function among 
meromorphic functions of order $1$, but Wieland's result does not extend to higher Gamma function that are of higher order than $1$.


This new definition presents some advantages compared to other more popular ones. It is natural from the 
point of view of complex analysis and  
generalizes naturally to higher Gamma functions. Before presenting these more general 
results, we review in the next section the convoluted history of all the different 
definitions and characterizations of Euler Gamma function. In a third section we examine the equivalence of our 
definition with other existing definitions and characterizations, and we derive from it a substantial amount of 
classical results for Euler Gamma function. In the companion article \cite{PM2020}, 
we generalize our definition to higher Gamma functions.

\section{Historical definitions of Euler Gamma function. }

The literature 
on Euler Gamma function is composed by literally thousands of pages, and covers 
almost 300 years of articles and monographs in Latin, English, French, German,... 
Unfortunately many citations 
and references are erroneous, even from Gauss or Weierstrass. In particular, many  of Euler's contributions 
have been misattributed. Euler himself also forgets to give proper credit to some of its predecessors.  
We have tried, to the best of our knowledge,
to trace to the original sources. Some notes with corrections to the literature can be found in \cite{PM1}.

For a brief historical accounts on the subject of the Gamma function and Eulerian Integrals (as the subject 
was known in the XIXth century), 
we refer to the accurate historical notes in chapter 2 of Remmert's book (1998, \cite{Rem2}).  
The recent Aycock's preprint (2019, \cite{Ayc3}) elaborates on Euler's original contributions. 
Davis' article (1959,  \cite{Da}) is informative but lacks of precise references. Other useful 
more recent articles are Sebah and Gourdon (2002, \cite{SG}), Srinivasan (2007. \cite{Sri}), Borwein and 
Corless (2017,  \cite{BC}).

For work before the XXth century we refer Godefroy's book (1901,  \cite{Go}) and to Nielsen's book (1906,  \cite{Nie}) and 
its bibliography. Also Gronwall's english translation of Jensen's article  (1916, \cite{Je1916}) is informative. 
In particular the memoir by Godefroy (1901,  \cite{Go}) and the article by Pringsheim (1888, \cite{Pr}) 
are particularly focused in the various 
definitions before the XXth century.

\medskip

\subsection{Euler integral form definition}
L. Euler (1730) introduced in a letter to  Goldbach, dated January 8th 1730 (\cite{[E00717]})
an integral form for the Gamma function that 
interpolates the factorial function valid for positive real  
values $s>0$,
\begin{equation}\label{eqn:property1}
\Gamma(s) = \int_0^1 \left (\log (1/t)\right )^{s-1} dt  \ .
\end{equation}
The interpolation of the factorial and associated series was started by J. Stirling (1730, \cite{Sti}), and this integral 
representation was a major progress by Euler. 

The most popular version nowadays of this integral form is:
\begin{equation}\label{eqn:property2}
\Gamma(s)=\int_0^{+\infty} t^{s-1} e^{-t}\, dt \ .
\end{equation}

Euler Gamma function is introduced by means of this formula and its properties developed from there in 
many contemporary expositions. Some classical authors that have prefered this approach 
include Legendre, Liouville, Pringsheim,...In particular Legendre's ``Exercices de Calcul integral'' \cite{Le} 
was one of the main references for Eulerian Integrals in the XIXth century.

\subsection{Product formula definition}\label{sec:product-def}

Euler originally defines the Gamma function by a product formula in a letter to Goldbach, dated 
October 13th 1729  \cite{[E00715]}, 
\begin{equation}\label{eqn:property3}
\Gamma (s)= \lim_{n\to +\infty} \frac{n!}{s(s+1)\ldots (s+n)} \,  n^s  \, .
\end{equation}
Later, this product formula was taken by Gauss (1812, \cite{Ga}) as the starting point to develop the theory.
Apparently, Gauss was unaware of Euler's letter to Goldbach and didn't credit him for the product formula, neither 
Weierstrass was aware and credited Gauss for it, even as late as 1876 in \cite{Wei3} (as remarked by Remmert, 
\cite{Rem2} p.41).
This approach is more direct and has some advantages. In the introduction of \cite{Go}, 
Godefroy discuss its merits compared to the integral form definitions.

\subsection{Weierstrass product definition}\label{sec:W-product-def}

Weierstrass (1876, \cite{Wei2}) developed a general theory of factorization of meromorphic functions in 
the complex plane, later made precise by Hadamard for functions of finite order. In particular, we have
the factorization of Euler Gamma function
\begin{equation}\label{eqn:property4}
\frac{1}{\Gamma (s)}= s e^{\gamma s} \prod_{n=1}^{+\infty } \left (1+\frac{s}{n}\right ) e^{-s/n} 
\end{equation}
that was known before Weierstrass. Davis \cite{Da} gives credit to Newman (1848, \cite{Ne}) that found this factorization, 
but N\"orlund \cite{No1} and Remmert \cite{Rem2} attribute the formula 
to Schl\"omich (1843, \cite{Sch}) that found it earlier. Aycock observed in \cite{Ayc3} p.21 that the computations 
to reach the product representation are already contained in Euler's work (1787, \cite{[E613]}).

This Weierstrass factorization can be taken as definition of Euler Gamma function. Obviously, it contains 
much information (for instance, the divisor structure) that is granted without prior justification. Nevertheless, 
this is sometimes the preferred approach 
in some of the reference literature on special functions, as for example Whittaker and 
Whatson (1927,\cite{WW}).

\subsection{Birkhoff's asymptotic expansion definition.}

Birkhoff (1913, \cite{Bir}), inspired from Stirling's asymptotics, gave an unorthodox definition of Euler Gamma function as a limit, for $s\notin -\NN$,
$$
\Gamma(s) = \lim_{n\to +\infty} \frac{\varphi(x+n+1)}{s(s+1)\ldots (s+n)} 
$$
where 
$$
\varphi (s)=\sqrt{2\pi} s^{s-1/2}e^{-s} \ .
$$
He proves directly that this limit exists and is a meromorphic function without 
zeros. Birkhoff derived from its definition some of the main classical results for 
the Gamma function. Similar to Weierstrass product definition, this assumes the a priori knowledge 
of Stirling formula and it is not very natural for an introduction from scratch of Euler Gamma function.

\subsection{Inverse Laplace transform and Hankel contour integral definitions.}\label{sec:Hankel-def}

The inverse Laplace transform formula, found by Laplace (1782, \cite{Lap1}) is probably the oldest complex analytic formula 
for (the inverse of) the Gamma function, and probably for this reason is well deserving of the qualification of ``r\'esultat remarquable'' 
as Laplace refers to it in \cite{Lap1} and \cite{Lap2}. Laplace formula is, for $\Re s >0$ and  $x>0$,
$$
\frac{1}{\Gamma (s)} = \frac{1}{2\pi} \int_{-\infty}^{+\infty} \frac{e^{x+iy}}{(x+iy)^s} \, dy
$$
The integral is independent of $x>0$. One can define the Gamma function from this formula and derive its main properties as is
done by Pribitkin  \footnote{The article starts with the odd historic claim that ``In 1812 Laplace establishes...'' 
citing the ``Trait\'e analytique des probabilit\'es'' \cite{Lap2} p.134 but the formula was established 30 years 
earlier in 1782 \cite{Lap1}.} (2002, \cite{Pri}).

\medskip

The Laplace transform formula is closely related to the Hankel contour integral representation for the inverse of the Gamma 
function found by Hankel  (1864, \cite{Hankel}). Consider a path $\eta$ from $-\infty$ to $-\infty$ surrounding the negative real axes 
$\RR_-=]-\infty, 0]$ (see Figure 1), then we have for all values of $s\in \CC$
$$
\frac{1}{\Gamma (s)} =\frac{1}{2\pi i} \int_\eta z^{-s} e^z \, dz
$$
where for $z\in \CC-\RR_-$,  $z^{-s}=e^{-s\log z}$ taking the principal branch of the logarithm function. Note that Hankel integral 
converges without any restriction on $s\in C$ and thus defines an entire function.

\medskip

\begin{figure}[ht]
\centering
\resizebox{8cm}{!}{\includegraphics{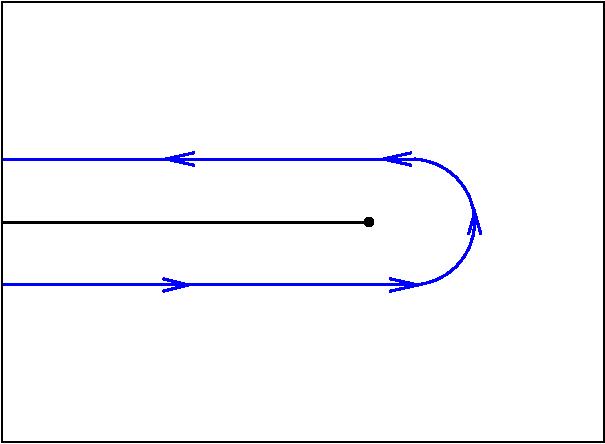}}    
\put (-20,150) {$\CC$}
\put (-130,115) {$\eta$}
\put (-190,90) {$\RR_-$}
\put (-130,45) {$\eta$}
\put (-85,85) {$0$}
\caption{Hankel contour.} 
\end{figure}
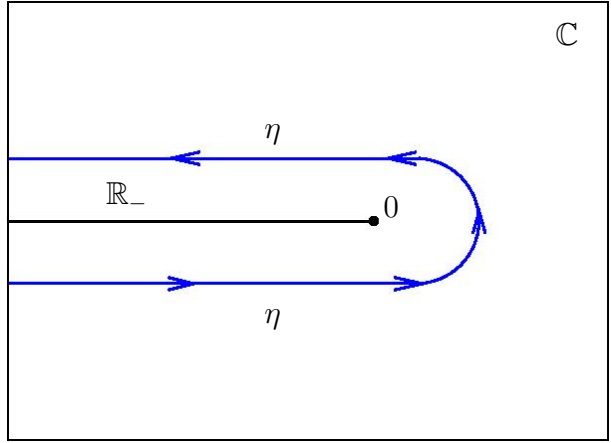

\medskip

It is not difficult to see, that for $\Re s >0$, the Laplace integral is equal to Hankel integral.
For $R>0$, consider $\eta_R=\eta \cup \{\Re z >-R\}$ and complete it in 
a closed polygonal contour $C_R$ with vertices at $-R\pm iR$, $x \pm iR$, see Figure 2.   
We have by Cauchy formula
$$
\int_{C_R} z^{-s} e^z \, dz =0
$$
and 
$$
\lim_{R\to +\infty} \int_{C_R} z^{-s} e^z \, dz = \int_\eta z^{-s} e^z \, dz + i \int_{-\infty}^{+\infty} \frac{e^{x+iy} }{(x+iy)^s}  \, dy
$$
thus
$$
\frac{1}{2\pi i} \int_\eta z^{-s} e^z \, dz = \frac{1}{2\pi} \int_{-\infty}^{+\infty} \frac{e^{x+iy}}{(x+iy)^s} \, dy  \ .
$$

\medskip

\begin{figure}[ht]
\centering
\resizebox{8cm}{!}{\includegraphics{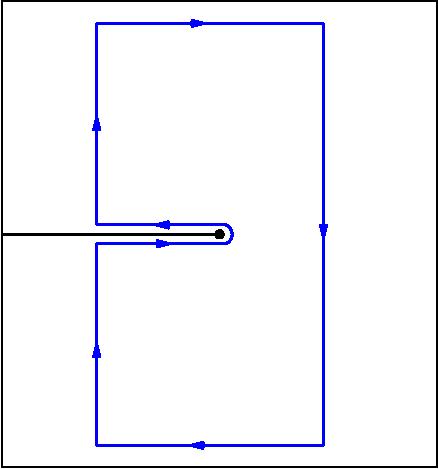}}    
\put (-130,217) {$C_R$}
\put (-210,125) {$\RR_-$}
\put (-130,20) {$C_R$}
\put (-105,110) {$0$}
\put (-226, 230) {$-R+iR$}
\put (-55, 230) {$R+iR$}
\put (-55, 5) {$R-iR$}
\put (-226, 5) {$-R-iR$}
\put (-55,120) {$C_R$}

\caption{Completed polygonal contour.} 
\end{figure}
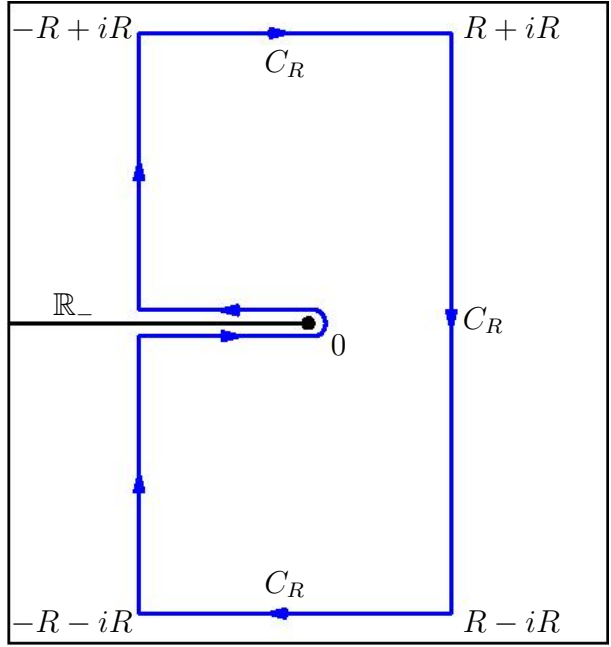

\medskip

There is a related second Hankel integral formula, for $s\in \CC-\NN_-$, 
$$
\Gamma (s) =\frac{1}{2i\sin (\pi s)} \int_\eta z^{s-1} e^z \, dz \ .
$$
Both Hankel integrals are related by Euler reflection formula
$$
\Gamma (s) \Gamma (1-s) =\frac{\pi}{\sin(\pi s)} \ .
$$

\subsection{Lerch derivation from Hurwitz zeta function}

Euler Gamma function can be defined from Hurwitz zeta functions via Lerch formula (1894, \cite{Le1894}) 
\begin{equation}\label{eqn:property5}
\log \Gamma(s) = \left [\frac{\partial}{\partial t}\zeta(t,s) \right ]_{t=0}-\zeta'(0)
\end{equation}
where 
\begin{equation}\label{eqn:property5-def}
\zeta(t,s)=\sum_{k=0}^{+\infty} \frac{1}{(s+k)^t}
\end{equation}
is the Hurwitz zeta function, $\Re t>1$ and $s\in \CC$, and $\zeta(t)=\zeta(t,1)$ is the usual Riemann zeta 
function. 

This requires some previous work by proving the meromorphic 
extension of Hurwitz zeta function to a neighborhood of $t=0$. This is usually done through a Hankel  
contour integral following closely what was done by Riemann for the Riemann zeta function in his memoir 
on the distribution of prime numbers \footnote{Despite the terminology 
Riemann performed his contour integral some years before Hankel. It is difficult to trace the original 
source for Hankel type contour integrals.} (1859, \cite{Rie}). 
This approach has been used very effectively by Barnes and Shintani for the definition of higher Gamma functions.
Unfortunately, it does not make 
much sense to define Euler Gamma function since it 
can itself be defined directly with such a Hankel integral, as seen in the 
previous section. This approach only makes sense provided that we prove the meromorphic 
extension of Hurwitz zeta function to $t=0$
by other means. Note also that we are not allowed to use the  integral formula
$$
\zeta(t,s) =\frac{1}{\Gamma (s)} \int_0^{+\infty} \frac{t^{s-1} e^{-at}}{1-e^{-t}} \, dt
$$
that contains the Gamma function (there is an elementary proof of the meromorphic extension 
using this integral formula).  Berndt (1985, \cite{Be1985}) presented a direct way, on which he bases his derivation 
of the main properties of Euler Gamma function. 
By elementary integration by parts, or using Euler-McLaurin formula to the first order, we get
$$
\zeta(t,s) =\frac{s^{1-t}}{t-1} + \frac{s^{-t}}{2} -t \int_0^{+\infty} \frac{u-[u]-1/2}{(u+s)^{t+1}} \, du
$$
which is valid for $\Re t >-1$ and proves the meromorphic extension to $t=0$.

\subsection{Functional equation}

The condition $\Gamma (1)=1$ and the functional equation
$$
\Gamma(s+1)=s\Gamma(s)
$$
determine the values taken at the positive integers  and 
show that the Gamma function is an interpolation of the factorial function, for $n\in \NN$,
$$
\Gamma(n+1) =n!
$$
Unfortunately, these conditions do not suffice to determine $\Gamma(s)$ uniquely for 
any real or complex argument $s$, since we can always multiply a solution by the exponential of a non-zero periodic function 
vanishing at the positive integer values. Euler tried to solve the functional equation 
by different methods, not always rigorous but displaying a remarkable creativity. For example,  
by using the Euler-McLaurin formula or by formulating the difference equation as an infinite order differential equation 
(see the fascinating account by Aycock (2019, \cite{Ayc3})).

Starting from the XIXth century, mathematicians from the German school added conditions to the functional 
equation to determine the Gamma function uniquely.

\subsection{Functional equation plus asymptotic estimates}

Weierstrass (1856,\cite{Wei})  added to the functional equation 
the condition\footnote{Aycock \cite{Ayc3} p.18 remarks that Weierstrass added the 
condition $\Gamma(1)=1$ which is not needed.}:
\begin{equation}\label{eqn:property6}
\lim_{n\to +\infty} \frac{\Gamma(s+n)}{(n-1)!\, n^s} =1
\end{equation}
Weierstrass attributes this last condition to Gauss, but as observed by Aycock \cite{Ayc3} p.24, it was
already known to Euler (postumosly published in 1793,\cite{[E652]}). See also the article by Prym (1877, \cite{Pry}). Iterating 
the functional equation, and using this asymptotic estimate
leads directly to the Euler-Gauss product formula.

\medskip

Essentially equivalent is the characterization by   
Laugwith and Rodewald (1987, \cite{LauRod1987}) who prove, assuming $\Gamma(s)>0$ for real $s>0$, 
$\Gamma (1)=1$, that it is enough to add to the functional equation the limit
\begin{equation}\label{eqn:property7}
\lim_{n\to +\infty} \left (\log \Gamma(s+n) -\log \Gamma (n)- s\log (n+1) \right ) =0
\end{equation}
in order to determine the Gamma function uniquely. Limit (\ref{eqn:property7}) is obtained from  (\ref{eqn:property6}) 
using $\Gamma (n)=(n-1)!$ and taking the logarithm.

\subsection{Functional equation plus estimates: Wielandt's, Fuglede's and Smith's characterization}

In 1939 Wielandt (\cite{Wie}, \cite{Rem1}) observes, by using Liouville Theorem, that there is only one function 
satisfying the functional equation, holomorphic in the right half plane $\CC_+$, taking the value $1$ 
at $1$, and bounded on a vertical strip of width larger than $1$. More precisely, he proves:

\begin{theorem}[Wielandt, 1939]
Let $f$ that is holomorphic in $\CC_+$, $f(1)=1$, bounded in the strip $S_1=\{1\leq \Re s <2\}$, such that 
$$
f(s+1)=sf(s)
$$
then $f=\Gamma$.
\end{theorem}

Liouville Theorem was used 
before by Birkhoff (1913, \cite{Bir}) in order to give a quick proof of  
some properties of Euler's Gamma function from his definition. The arguments are 
close to those in Wielandt's characterization (see also Remmert's comments at this respect in \cite{Rem1}). It is possible 
that Wieland was inspired by Birkhoff's article.

From Euler integral form of the Gamma function we deduce easily  that it is bounded on vertical strips in the 
right half plane $\CC_+$, hence we get a very simple and satisfactory characterization 
of Euler Gamma function\footnote{As observed before, we also get a definition, but this 
is not the approach taken by Remmert in  \cite{Rem1} where he assumes the Gamma function is initially defined by Euler's integral formula}. 
Note also, that  Wielandt's Theorem does not even need to asume the function to be real analytic.
We have a full genuine complex characterization.

See Remmert's article \cite{Rem1} for a detailed account 
(see also Aycock \cite{Ayc3} p.12).

\medskip

Fuglede (2008, \cite{Fu}) weakens substantially Wielandt's boundedness condition by 
noticing that we can allow exponential growth (of order $1$) in the vertical strip. More precisely, he proves that we can replace boundedness of $f$ on $S_1$ by a weaker condition.

\begin{theorem}[Fuglede, 2008]
In Wielandt's theorem we can replace boundedness of $f$ on $S_1$ by the condition, for $s\in \overline{S}_1$ 
$$
f(s)=\cO \left (|\Im s|^{\Re s-\frac{1}{2}} e^{\frac{3\pi}{2} \Im s}\right )
$$
\end{theorem}

To be precise, Fuglede gives three weaker more technical conditions that implie this one, and proves that this condition is 
also optimal. In his paper Fuglede makes use of Stirling asymptotic formula and the reflection formula for the Gamma function.

\medskip

In an apparently unpublished manuscript, Smith (2006, \cite{Sm}) proves  the following characterization of Euler Gamma function:

\begin{theorem}[Smith, 2006]
 Let $f$ be a meromorphic function on $\CC_+$ such that $1/f(s)=\cO (e^{c |s|})$, with $c<\pi$, and 
$$
f(s+1)=sf(s)
$$
then $f=\Gamma$.
\end{theorem}

Smith states the result with $c=\pi$, but gives an sketchy proof using a Theorem of Carlson that requires $c<\pi$.
Note that the bound forbids $f$ of having zeros in $\CC_+$. This characterization is similar to the previous ones in a vertical strip, 
weaker than Weilandt's, but stronger than Fuglede's. 
The proof is based on Carlson's Theorem that we will review later.

\medskip

Our definition is closer in the spirit to these characterizations and  weakens in an important way  
Wielandt, Fuglede and Smith bounds by allowing a growth of finite order function (i.e. 
as exponential of a polynomial) in the vertical strip, in the positive half plane $\CC_+$, or in a cone around the 
positive real axes $\RR_+$. Fuglede's estimates being optimal, we need to add some extra assumptions. 
In our case we assume that $\Gamma'(1)\in \RR$, and that 
the divisor is confined to $\CC-\CC_+$.
Note that Wielandt's condition includes that there are no poles in the right 
half plane. Our condition is somewhat more  symmetrical from the point of view of 
meromorphic functions. Indeed, we will see that Wielandt's and Fuglede's estimates prove that the function has no zeros 
on the right half plane (for Smith estimate it is straightforward), hence their bounds implies our condition on the divisor.

\subsection{Functional equation plus log-convexity: Bohr-Mollerup's characterization}

H. Bohr and J. Mollerup prove (as an exercice in his Calculus book \cite{BoMo}, 1922) 
that the real Gamma function is characterized by the functional equation and log-convexity, more precisely:

\begin{theorem}[Bohr-Mollerup, 1922]
Let $f:\RR_+\to \RR_+$ be a positive function, $f(1)=1$, such that $\log f$ is convex and 
$$
f(x+1)=xf(x)
$$
then $f=\Gamma_{/\RR}$.
\end{theorem}

We get in this way a perfectly satisfactory characterization from the point of view of purely real analysis.

E. Artin made of Bohr-Mollerup theorem the starting point of his theory 
of the Gamma function that is presented in his monograph (1931, \cite{Ar}). Artin solved the difficult problem of the 
intricate historical bibliography by not citing any of his predecessors. To prove that the unique solution is 
real analytic and extends to a meromorphic function on $\CC$ one needs to do some extra work. Artin proves the 
existence part and the analytic extension by using Euler's integral formula and proving the log-convexity. Hence, 
his approach is not really independent from Euler's integral definition.
In the extension to higher Barnes Gamma 
functions, also characterized by log-convexity by Vign\'eras (1979, \cite{Vi}), the existence and analytic extension is 
proved by identifying their Weierstrass factorization. These authors could have taken directly these other 
more analytic definitions, since after all they need them to justify the real analytic character of the solution.

\medskip

Bohr-Mollerup  approach, by its elementary nature, became popular among non-complex analysts, in 
particular among algebrists, but it does not feel to be  
the natural point of view from complex analysis. Remmert makes this point  
in his article \cite{Rem1} and book \cite{Rem2} when presenting Wieland's characterization, 
that, he though of as the more natural from the complex analyst point of view (he mentions also the article of 
D. Birkhoff expressing similar views). Bohr-Mollerup approach requires the real analyticity condition that 
is not needed in purely complex characterizations. Moreover, the positivity assumption implicitly requires 
that no zeros, nor poles lie on the positive real axes. 

\medskip

Bourbaki (1949, \cite{Bou}, Chapter VII) adopts  also this approach to characterize the Gamma function.
One cannot dispute that 
the complex extension of the Gamma function is fundamental, hence, from a categorical point of view, dear 
to algebrists (pun intended), we 
should seek a natural complex analytic definition. Remmert in the Historical Introduction of 
his ``Theory of complex functions'' \cite{Rem3}
makes this point by reminding two memorable early quotes from Gauss that is worth recalling. In a letter 
to Bessel dated December 18th 1811 
(see \cite{Rem3} p.1, \cite{Ga2})
\medskip

\textit{``At the very beginning I would ask anyone who wants to introduce a new function into analysis 
to clarify whether he intends to confine it to real magnitudes (real values of its argument) and regard 
the imaginary values as just vestigial - or whether he subscribes to my fundamental proposition that in the 
realm of magnitudes the imaginary ones $a+b\sqrt{-1} =a+bi$ have to be regarded as enjoying equal rights 
with the real ones. We are not talking about practical utility here; rather analysis is, to my mind, 
a self-sufficient science. It would lose immeasurably in beauty and symmetry from the rejection of any 
fictive magnitudes. At each stage truths, which otherwise are quite generally valid, would have to be 
encumbered with all sorts of qualifications...``}

\medskip

Gauss also writes (according to Remmert sometime after 1831, \cite{Rem3} p.2)

\medskip

\textit{``Complete knowledge of the nature of an analytic function must also include insight into
its behavior for imaginary values of the arguments. Often the latter is indispensable even for a proper 
appreciation of the behavior of the function for real arguments. It is therefore essential that the 
original determination of the function concept be broadened to a domain of magnitudes which includes
both the real and the imaginary quantities, on an equal footing, under the single designation complex 
numbers.``}

\medskip

These quotes illustrate very accurately Remmert's reservations regarding Bohr-Mollerup definition.

\subsection{Functional equation plus more symmetries}

Davis \cite{Da} p.867 writes, without providing references,

\medskip

\textit{``By the middle of the 19th century it was recognized that Euler's gamma function was the only 
continuous function which satisfied simultaneously the recurrence relationship, the reflection formula 
and the multiplication formula.''}

\medskip

We can find such a characterization in Remmert's book \cite{Rem2}, p.46,

\begin{proposition}
Suppose that $F$ is a meromorphic function in $\CC$, positive in $\RR_+^*$, and satisfies the functional equation
$$
F(s+1)=sF(s)
$$
and the duplication formula
\begin{equation}\label{eqn:property8}
\sqrt{\pi} F(2s) = 2^{2s-1} F(s)F(s+1/2) \ .
\end{equation}
Then $F=\Gamma$.
\end{proposition}

There is a similar result for real functions in chapter 4 in Artin's book \cite{Ar}, and in 
chapter 6 a discussion on this type of characterizations.
Also in Hermite's lessons from 1882, \cite{He}, we can find a similar  characterizations.
They are interesting, but not very ``economical'' since they require a non-trivial 
second functional equation.

\subsection{N\"orlund fundamental solution to difference equations}

N\"orlund build a theory of analytic resolution of difference equations (see \cite{No1} and \cite{No3}) 
that can be used to define $\log \Gamma(s)$ as the ``fundamental solution'' (in N\"orlund's terminology) in 
the right half plane $\{\Re s >0\}$ of the difference equation
$$
\Delta \log \Gamma(s) =\log s
$$
N\"orlund's theory is interesting, but elaborated and not suitable for a direct approach.
It constructs directly the general fundamental solution. This approach, to build directly the fundamental solution
has its origins in the 
work of Euler and his use of Euler-McLaurin formula (see \cite{[E613]} and 
the discussion in \cite{Ayc3} p.19). N\"orlund's approach 
has been continued more recently by Ruijsenaars in his interesting article \cite{Rui1997}.

\subsection{Scope of the various definitions}

The reference for one or another definition depends very much on the approach we take to the theory 
of Euler Gamma function. From the modern point of view, definitions based on characterizations by important 
properties as  the functional equation have prevailed. From the complex analytic point of view, Wielandt 
 definition
is satisfactory and more suitable than the Bohr-Mollerup approach.

\medskip

It is also important, in order to choose the best definition, that the main properties of the 
Gamma function, as well as alternative characterizations, can be derived easily. 
For example, Remmert takes good care in \cite{Rem1} and \cite{Rem2} to derive using Wielandt's 
characterization the main properties of Euler Gamma function. In the next section, we will do the same 
with our definition.

\medskip

One drawback of Wielandt's definition, and maybe the reason it has been unfairly ignored, 
is that it does not  carries over to higher 
Gamma functions. We are particularly interested in higher Gamma functions defined by Barnes. Already in \cite{Ba}, 
Barnes observes that the functional equation plus asymptotic estimate \`a la Weierstrass characterizes 
Barnes $G$-function. Higher Barnes Gamma functions and multiple Gamma functions of Barnes 
are derived \textit{\`a la Lerch} by Barnes from their higher Hurwitz 
zeta functions. This requires the development of the theory for these functions first. The Bohr-Mollerup 
definition is used and generalized by Vign\'eras \cite{Vi} to characterize higher Barnes Gamma functions, but 
for the analytic extension needs to come back to their 
Weierstrass factorization.

\medskip

Our definition of Euler Gamma function from Theorem \ref{thm:main-Euler-gamma} generalizes, with minimal complex analytic machinery,
to higher Gamma functions, including Mellin's Gamma functions, Jackson's $q$-Gamma function, Barnes multiple Gamma functions and Shintani Gamma functions as is
done in the companion article \cite{PM2020}.

\medskip

The reason that explains the success of our definition is that 
it works naturally for the space of finite order meromorphic function. From the transalgebraic point of view, 
finite order meromorphic functions is the natural class of transalgebraic 
functions. It comes without surprise to see  
this definition appearing naturally in the related transalgebraic work \cite{MPM1} (see also \cite{MPM}).

\subsection{On definition \textit{vs.} characterization.}

We point out briefly the subtle difference between a \textit{definition} and a \textit{characterization} of the Gamma function. 
A characterization implicitly assumes that the Gamma function has been defined previously, and we 
prove a uniqueness theorem for some of its properties. A definition does not presupposes that the Gamma function 
has been defined, hence we aim to derive the whole theory from it. Obviously, a characterization can always be turned 
into a definition, but requires the extra work to prove the existence. Also we can use the characterization to develop 
the theory, but then the original definition does not need to be proved and is skipped. This can lead to easier proofs. 
Our point of view is minimal and we focus on definitions from where we develop the whole theory. For each adopted definition there
is a proper order for the proofs.

\section{Derivation of the main properties of Euler Gamma function.}

We derive from our definition the other characterizations or definitions, and the main properties of Euler 
Gamma function.

\medskip

$\bullet$ \textbf{Weierstrass product.}
\medskip

We  already derived in Corollary \ref{cor:W-product} the Weierstrass product 
(section \ref{sec:W-product-def}) for our Gamma function,
$$
 \Gamma(s) =s^{-1}e^{-\gamma s}\prod_{k=1}^{+\infty} \left ( 1+\frac{s}{k}\right )^{-1}e^{s/k} \ .
$$

Conversely, we can define, with Weierstrass,
$$
f(s)=s^{-1}e^{-\gamma s}\prod_{k=1}^{+\infty} \left ( 1+\frac{s}{k}\right )^{-1}e^{s/k} \ .
$$
Then $f$ has no zeros nor poles on $\CC_+$, it is a meromorphic function of finite order (order one),
$f(s+1)=sf(s)$ (by the same computation from the remarks in the introduction), hence 
$$
f(s+1)=e^{-\gamma s}\prod_{k=1}^{+\infty} \left ( 1+\frac{s}{k}\right )^{-1}e^{s/k}
$$
so $f(1)=f(0+1)=1$, and 
$$
\frac{f'(s+1)}{f(s+1)}=-\gamma  +\sum_{k=1}^{+\infty} \frac{1}{k} -\frac{1}{k+s}
$$
so, making $s=0$, $f'(1)=f'(1)/f(1) =-\gamma \in \RR$.

Therefore, the hypothesis of Theorem \ref{thm:main-Euler-gamma} are satisfied and $f=\Gamma$.

\medskip

$\bullet$ \textbf{Product formula.}

\medskip

The product formula definition from section \ref{sec:product-def} follows for our Gamma function from the Weierstrass product 
by a simple computation:

\begin{corollary}[Product formula]\label{cor:product_formula}
For $s\in \CC$, we have
$$
\Gamma(s)=\lim_{n\to +\infty} \frac{n! \,  n^s }{s(s+1)\ldots (s+n)} \ .
$$
\end{corollary}
\begin{proof}
We compute,
\begin{align*}
\Gamma(s) &= s^{-1}e^{-\gamma s}\prod_{k=1}^{+\infty} \left ( 1+\frac{s}{k}\right )^{-1}e^{s/k} \\
&=\lim_{n\to +\infty} s^{-1} \exp \left ( -\left (\sum_{k=1}^n \frac{1}{k} -\log n \right )s 
-\left (\sum_{k=1}^n \log(1+s/k) \right )+\left (\sum_{k=1}^n \frac1k \right ) s \right )\\
&=\lim_{n\to +\infty} s^{-1} n^{s} n!\prod_{k=1}^n (s+k)^{-1} \\
&=\lim_{n\to +\infty} \frac{n!}{s(s+1)\ldots (s+n)} \,  n^s 
\end{align*}
\end{proof}

Conversely, we can define, with Euler and Gauss,
$$
f_n(s)=\frac{n!}{s(s+1)\ldots (s+n)} \,  n^s \ .
$$
and
$$
f(s)=\lim_{n\to +\infty} f_n(s) \ .
$$
It is easy to check that the product converges uniformly on compact sets of $\CC-\NN_-$ to a holomorphic function without zeros. Also
$$
f_n(s+1)=\frac{n}{n+1}s f_n(s)
$$
hence when $n\to +\infty$ we get $f(s+1)=sf(s)$. Also $f_n(1)=\frac{n}{n+1}$ so $f(1)=1$. Also we have $f'(1)\in \RR$ since $f$ is real 
analytic because $\overline{f_n(\bar s)}=f_n(s)$.
We  check that $f$ is of finite order to conclude using Theorem \ref{thm:main-Euler-gamma} that $f=\Gamma$.
We have 
\begin{lemma}
 Let $0<a<b<+\infty$, and consider the vertical strip 
$S_{a,b} =\{s\in \CC ; a\leq \Re s \leq b\}$. Then
$f_n$ is bounded in $S_{a,b}$, and more precisely,
$$
||f_n||_{C^0(S_{a,b})} \leq ||f_{n /\RR}||_{C^0([a,b])}
$$
\end{lemma}
\begin{proof}
For $s\in S_{a,b}$,  $\Re s=x \in [a,b]$, we have $|s(s+1)\ldots (s+n)| \geq x(x+1)\ldots (x+n)$. Therefore we have
$$
|f_n(s)|=\left |\frac{n! \, n^s }{s(s+1)\ldots (s+n)}   \right | \leq \frac{n! \, n^x }{x(x+1)\ldots (x+n)}  =|f_n(x)|
$$
\end{proof}

Using the uniform convergence on $[a,b]\in \RR_+^*$ we get: 
\begin{corollary}
Let $0<a<b<+\infty$, and consider the vertical strip 
$S_{a,b} =\{s\in \CC ; a\leq \Re s \leq b\}$. Then
$f$ is bounded in $S_{a,b}$, and more precisely,
$$
||f||_{C^0(S_{a,b})} \leq ||f_{/\RR}||_{C^0([a,b])}
$$
\end{corollary}

Using as region $\Omega$ in Proposition \ref{prop:domain} a vertical strip of width larger than 1, we get that $f$ is of finite order.
Hence we have $f=\Gamma$ and in particular we have proved:

\begin{corollary}\label{cor:bounded}
Let $0<a<b<+\infty$, and consider the vertical strip 
$S_{a,b} =\{s\in \CC ; a\leq \Re s \leq b\}$. We have that 
$\Gamma$ is bounded in $S_{a,b}$, and more precisely,
$$
||\Gamma||_{C^0(S_{a,b})} \leq ||\Gamma_{/\RR}||_{C^0([a,b])}
$$
\end{corollary}

\medskip

$\bullet$ \textbf{Wielandt characterization.}

\medskip

We review carefully Wielandt's characterization that is closer to our definition. We recall the proof 
of Wielandt's theorem (the proof below is a slight simplification of Remmert's proof from \cite{Rem1}).

\begin{theorem}[Wielandt, 1939]
There is at most a unique function $f$ that is holomorphic in $\CC_+$, $f(1)=1$, bounded in $\overline{S}_{1,2}$, such that 
$$
f(s+1)=sf(s)
$$
\end{theorem}

\begin{proof}
Any such function $f$ extends meromorphically to the whole complex plane using the 
functional equation. Observe also that $f$ is bounded in any strip $S_{a,b}$, for $0<a<a+1<b$, using a 
finite number of times the functional equation (the bound depends on $a$ when $a\to 0$).

Let $f$ and $g$ be two solutions. 
From the functional equation they have the same simple poles at the negative integers, with the same residues, for $n\geq 0$,
$$
\Res_{-n} f = \Res_{-n} g = \frac{(-1)^n}{n!}
$$
The difference $h(s)=f(s)-g(s)$ is an entire function satisfying the functional equation, 
and it is bounded in the strip $\overline{S}_0=\overline{S}_{0,1}$ because it is 
bounded in the compact set $\overline{S}_0 \cap \{|\Im s| \leq 1\}$, and for $|\Im s| \geq 1$ we use $h(s)=s^{-1} h(s+1)$. 
If $l(s)=h(s)h(1-s)$, we have $l(s+1)=-l(s)$ and $|l|$ is $1$-periodic. So
$l$ is bounded on the plane and must be constant by Liouville Theorem. 
But $l(1)=0$ thus $l$ is identically $0$ and $f=g$.
\end{proof}

\medskip

We note that what is proved is only the uniqueness, but not the existence. Remmert in his article  \cite{Rem1} assumes that 
the Euler Gamma function has been previously  defined using Euler integral formula, and from there he checks  
the boundedness property and the other hypothesis.
But if we want to use Wielandt's theorem as the starting point of the theory of the Gamma function, we need to 
prove the existence of the solution \textit{without a prior definition of the Gamma function}. This is what is needed 
to go from ``Wielandt's characterization'' to ``Wielandt's definition''. This is done with the same arguments 
as given in the existence part of Theorem \ref{thm:main-Euler-gamma} as follows:

\begin{itemize}
 \item We build $f$ with a divisor compatible with the functional equation using the Weierstrass product (\ref{eq:W-solution}).
 \item We prove that $f$ is bounded on vertical strips following the path leading to Corollary \ref{cor:bounded}.
\end{itemize}

\medskip

Thus, from Corollary \ref{cor:bounded}, we get that our unique solution in Theorem \ref{thm:main-Euler-gamma} satisfies Wielandt's conditions, and it  is the unique solution in 
Wielandt's Theorem. Conversely, we show directly by elementary means in the next two Propositions 
that Wielandt's conditions implie the conditions 
in our main Theorem \ref{thm:main-Euler-gamma}. To prove, from Euler definition or from Wielandt's assumptions, that the Gamma 
function has no zeros in $\CC$ is not trivial and this was the subject of interesting discussions in the Hermite-Stieltjes correspondence (\cite{HS}, volume II) 
where we can find a direct proof by Stieltjes from Euler integral formula. We give below a proof from Wielandt's assumptions.
It is a sort of ``mini-Riemann Hypothesis'' when one understands that Gamma functions are level $0$ zeta functions. 

\begin{proposition}
 Let $f$ be a function satisfying the conditions of Wielandt's theorem. Then $f$ extends to a  finite order 
 real analytic meromorphic function, in particular $\Gamma'(1)\in \RR$. 
\end{proposition}

\begin{proof}
The finite order meromorphic extension follows by the same argument as 
in Proposition \ref{prop:domain}. The hypothesis of 
Wielandt's Theorem are invariant under conjugation, \textit{i.e.} the conjugate 
function $\overline{f(\bar s)}$ also satisfies the hypothesis, hence, by uniqueness, the solution is real analytic. 
\end{proof}

We can be more precise about the order, and also prove that $f$ has no zeros:

\begin{proposition}\label{prop:no_zeros}
 Let $f$ be a function satisfying the conditions of Wielandt's theorem. Then the analytic extension of $f$ 
 to a  finite order meromorphic function is of order $1$ and has no zeros on $\CC$. 
\end{proposition}

\begin{proof}
Under the given assumptions, we observe that multiplying  $f(s)f(1-s)$ by $\sin (\pi s)$ 
kills the poles that we know are simple because of the functional equation. Again the functional equation shows that 
$g_0(s)=f(s)f(1-s) \sin (\pi s)$ is $1$-periodic entire function. As before, we consider the closed strip 
$\overline{S}_0=\overline{S}_{0,1}$ of width $1$ and symmetric with respect to the vertical 
line $\{\Re s= 1/2 \}$, that is, invariant by $s\mapsto 1-s$. There is some constant $C_0>0$,
such that for $s\in \overline{S}_{0}$, $|\Im s| \geq 1$,
$$
|f(s)|=|s^{-1} f(s+1)|\leq C_0 |\Im s|^{-1} 
$$
and
$$
|f(s)f(1-s)|\leq C_0^2 |\Im s|^{-2} \ .
$$
Therefore, for $s\in \overline{S}_0$, $|\Im s| \geq 1$,
\begin{equation}\label{eq:estimate_g_0}
|g_0(s)|\leq C_0^2 |\Im s|^{-2} e^{\pi |\Im s|} 
\end{equation}
and by $1$-periodicity and the maximum principle we have (\ref{eq:estimate_g_0}) for all $s\in C$, thus $g_0$ and $f$ 
are of order $1$.

Now we prove that $f$ has no zeros and we argue by contradiction. 
Let $\rho\in \CC$ be a zero for $f$. Because of the functional equation $\rho \notin\ZZ$,   
all $\rho+\ZZ$ are zeros of $f$, and all $(1-\rho)+\ZZ$ are zeros of $f(1-s)$. Now we consider
$$
g(s)=\frac {g_0(s)}{\sin(\pi(s-\rho)) \sin(\pi(s-(1-\rho)))}=\frac {f(s)f(1-s)\sin \pi s}{\sin(\pi(s-\rho)) \sin(\pi(s-(1-\rho)))} \ .
$$
which is a $2$-periodic entire function. Observe that for $|\Im s|\geq 1$ there is a constant $C(\rho)>0$ such 
that we have the estimate
$$
\frac{1}{|\sin(\pi(s-\rho))|}\leq C(\rho) e^{-\pi |\Im s|}
$$
hence, combining with estimate (\ref{eq:estimate_g_0}), and using $1$-periodicity and the maximum principle,
this gives, for all $s\in \CC$, with $C_1(\rho )=C_0^2  C(\rho)  C(1-\rho)$,
$$
|g(s)|\leq C_1(\rho )|\Im s|^{-2} e^{-\pi |\Im s|}
$$
Using again $1$-periodicity and the maximum principle we get that $g$ is bounded and by Liouville theorem 
this implies that $g$ is constant, and since $g(s)\to 0$ when $|\Im s|\to +\infty$ we prove that $g$ is 
identically $0$. Contradiction.
\end{proof}

The last arguments applied to the Gamma function defined by Theorem \ref{thm:main-Euler-gamma} proves Euler's reflection formula:

\begin{theorem}[Reflection formula]
 The Gamma function satisfies the functional equation
 \begin{equation}
 \Gamma (s) \Gamma(1-s) =\frac{\pi }{\sin (\pi s)} 
 \end{equation}
\end{theorem}

\begin{proof}
We construct $g_0$ starting from $f=\Gamma$ as before. Observe that 
$$
g_0(s) =f(s)f(1-s)\sin(\pi s) =f(1+s)f(1-s)\frac{\sin(\pi s)}{s}
$$

We have shown that $g_0$ is of order $1$ and without zeros. Hence, $g_0(s)=e^{as+b}$. 
Taking logarithmic derivatives at $s=0$ in the above formula we get
$$
a=\frac{g_0'(s)}{g_0(s)} = \frac{f'(1)}{f(1)}-\frac{f'(1)}{f(1)} +0 = 0
$$
So $g_0$ is constant. Finally, using $f(1)=1$, we have
$$
g_0(s)=g_0(0)= \lim_{s\to 0} f(s+1) f(1-s) \frac{\sin(\pi s)}{s}=\lim_{s\to 0} \frac{\sin(\pi s)}{s}=\pi 
$$
\end{proof}

Making $s=1/2$ in the Reflection formula we get the value of $\Gamma(1/2)$. 

\begin{corollary}\label{cor:Gamma(1/2)}
 We have 
 $$
 \Gamma (1/2) =\sqrt{\pi} 
 $$
\end{corollary}

Take note of the remarkable fact that we did not need to compute a single integral in order to obtain the value of $\Gamma(1/2)$.

\medskip

We can see directly that a function $f$ satisfying Wielandt's conditions has no zeros in $\CC_+$ using 
Carlson's Theorem. Wielandt characterization is very close to Carlson Theorem. It is surprising that 
it is not mentioned in Remmert's analysis of Wieland theorem. We recall Carlson's Theorem \cite{Ca}.

\begin{theorem}[Carlson, 1914]
 Let $f:\CC_+\to \CC$ be a holomorphic function extending continuously to $\overline{\CC}_+$, of exponential type, 
 \textit{i.e.} such that for some $C_0, \tau >0$,
 $$
 |f(s)|\leq C_0 e^{\tau |s|} \ .
 $$
 We also assume that we have an more precise exponential control on the imaginary axes, for $y\in \RR$,
 $$
 |f(iy)|\leq Ce^{c\pi}
 $$
 where $c<\pi$.
 
 Then if $f(n)=0$ for $n\in \NN$ then $f$ is identically $0$, $f=0$. 
\end{theorem}

As the function $f(s)=\sin (\pi s)$ shows, we cannot take $c=\pi$.

Observe 
that such a zero $s_0$ gives a sequence of zeros $s_0+n$, $n\geq 0$, hence the function $g(s)=f(s-s_0)$ will vanish at the positive 
integers. If it satisfies the growth conditions in Carlson's Theorem, then $g=0$ which contradicts $g(1+s_0)=f(1)=1$. In order to complete 
the argument we need to establish first Stirling type asymptotics and derive a refinement of Carlson's Theorem like the one proved 
by Pila \cite{Pi}.

\medskip

$\bullet$ \textbf{Hankel integral formulas.}
\medskip

For Hankel integrals we recall that $z^s=e^{s\log z}$ in $\CC-\RR_-$ taking 
the principal branch of the logarithm function.

\begin{theorem}
We have
 $$
 \frac{1}{\G(s)} = \frac{1}{2\pi i} \int_\eta z^{-s}e^z \,  dz
 $$
 $$
\G(s)= \frac{1}{2 i \sin \pi s}\int_\eta z^{s-1} e^z \, dz
 $$
where $\eta$ is the path from Figure 1.
\end{theorem}

\begin{proof}
Let 
$$
f(s) = \frac{1}{2\pi i} \int_\eta z^{-s}e^z \,  dz 
$$
and we check that $1/f$ satisfies the conditions in Theorem \ref{thm:main-Euler-gamma}. 
We have the estimate,
$$
\left | z^{-s}e^z \right | \leq e^{\pi \Im s} |z|^{-\Re s} e^{\Re z}
$$
and the integral converges for all $s\in \CC$ and $f$ is an entire function. This estimate also 
proves that it is of order $1$ and more precisely that on any  vertical strip
$$
|f(s)|\leq C_0 e^{\pi \Im s}
$$
(the constant $C_0$ depends on the strip).

The differential form
$$
d\left (z^{-s} e^z\right ) =z^{-s}e^z dz -sz^{-(s+1)}e^z dz
$$
is closed, converges to $0$ when $\Im z \to -\infty$, so by integration 
over $\eta$ we get the functional equation $f(s+1)=s^{-1} f(s)$. Also, we have $\overline{f(\bar s)} =f(s)$ 
so $1/f$ is real analytic. The inverse function $1/f$ has no zeros, since $f$ is an entire function.

The only condition that remains to check is that $1/f$ is holomorphic in $\CC_+$, i.e. that 
$f$ has no zeros on $\CC_+$.

We proceed in a similar way as before. Observe that for $s\in \ZZ$, the 
function $z^{-s}$ is meromorphic and the residue formula gives that $f(n)=0$ for $n\in -\NN$ (the integrand 
is holomorphic), and for an integer $n\geq 1$, $f(n)=1/(n-1)!$. Hence the function 
$$
h_0(s)=\frac{f(s)f(1-s)}{\sin(\pi s)}
$$
is a $1$-periodic  entire function. Moreover, we have the estimate on vertical strips
$$
|h_0(s)|\leq C_0 e^{\pi \Im s}
$$
Now, if $f$ has a zero $\rho \in \CC$ then $\rho\notin \ZZ$, $\rho+\ZZ$ are zeros of $f$, $(1-\rho)+\ZZ$ are zeros of $f(1-s)$ and
$$
h(s)=\frac{h_0(s)}{\sin(\pi (s-\rho)) \sin(\pi (s-(1-\rho)))}
$$
is an entire function of period $1$ bounded in vertical strips as
$$
|H(s)|\leq C_0 e^{-\pi \Im s} \ .
$$
Liouville Theorem shows that  $h=0$ as before. So, $1/f$ is holomorphic on $\CC_+$ and 
we can use our main theorem to conclude $f=1/\Gamma$.

For the second integral, we use the reflection formula,
$$
\frac{1}{2 i \sin (\pi s)}\int_\eta z^{s-1} e^z \, dz =\frac{\pi \Gamma (1-s)^{-1}}{ \sin (\pi s)} =\Gamma (s) \ .
$$
\end{proof}

\medskip
$\bullet$ \textbf{Euler integral formula.}
\medskip

We prove now  Euler integral formula for the Euler Gamma function defined by our main Theorem.

%

\begin{theorem} For $\Re s>0$, we have
$$
\Gamma(s) = \int_0^1 \left (\log (1/u)\right )^{s-1} du  =\int_0^{+\infty} t^{s} e^{-t}\, \frac{dt}{t} \ .
$$
\end{theorem}

Note that the two integrals are equal by the change of variables $u=e^{-t}$.

\begin{proof} We give three proofs.

\medskip

$\bullet$ We can derive the integral formula from the already established product formula 
in Corollary \ref{cor:product_formula} by writing 
the product in integral form, for $\Re s>0$, $n\geq 1$,
$$
\frac{n! \,  n^s }{s(s+1)\ldots (s+n)}=\int_0^n t^{s-1} (1-t/n)^n\, dt
$$
This integral was known to Euler and it is easily proved by induction on $n\geq 1$ 
(see Hermite, 1884, \cite{He}, Chapter XV, p.140, or Gamelin, 2001, \cite{Gam},  chapter XIV, p.361). 
By dominated convergence, making $n\to +\infty$,  we get Euler integral.

\medskip

$\bullet$ We can also derive the theorem checking the conditions of our main Theorem.
We define for $\Re s>0$,
 $$
 f(s)=\int_0^{+\infty} t^{s-1} e^{-t}\, dt \ .
 $$
The function $f$ is holomorphic in $\CC_+$, hence has no poles in this half plane.
We have $f(1)= \int_0^\infty e^{-t} dt =1$ and $f$ is real analytic  so $f'(1)\in \RR$.
By integration by parts we get the functional equation,
$$
f(s+1)
= s f(s)
$$
and $f$ extends to a meromorphic function in $\CC$.
Also, we have 
$$
|f(s)|\leq \int_0^n t^{\Re s-1} (1-t/n)^n\, dt =f(\Re s)
$$
hence, $f$ is bounded in any strip $S_{a,b}$, $0<a<b$, and the same proof as in  Proposition \ref{prop:no_zeros} 
gives that $f$
is of order $1$ and has no zeros on $\CC_+$. We conclude that $f=\Gamma$.

\medskip

$\bullet$ We can start from the second Hankel integral and observe that when the path $\eta$ degenerates into the negative real axes,
we get Euler integral.
\end{proof}

\begin{corollary}
We have
$$
\int_0^{+\infty} e^{-x^2} \, dx =\sqrt{\pi} \ .
$$
\end{corollary}

\begin{proof}
By the change of variables $t=x^2$, we have that the integral is $\Gamma(1/2)$ and we use Corollary \ref{cor:Gamma(1/2)}.
\end{proof}

Note again that we succeed in computing the classical gaussian integral without any integral manipulations.

\medskip
$\bullet$ \textbf{Functional equation plus Weierstrass asymptotic.}
\medskip

The Weierstrass limit  (\ref{eqn:property6}) follows from the Euler-Gauss product (\ref{eqn:property3}) 
that we have already derived.

\begin{proposition}
We have
$$
\lim_{n\to +\infty} \frac{\Gamma(s+n)}{(n-1)!\, n^s} =1 \ .
$$
\end{proposition}
\begin{proof}
$$
\lim_{n\to +\infty} \frac{(s+n-1)\cdots (s+1)s}{(n-1)!\,  n^s} \, \Gamma(s) = \G(s)^{-1}\G(s)=1
$$
\end{proof}

Conversely, let $f$ be a meromorphic function satisfying the functional equation and Weierstrass asymptotic, 
$$
\lim_{n\to +\infty} \frac{f(s+n)}{(n-1)!\, n^s} =1 \ .
$$
Then, using the functional equation we have $f(n)=(n-1)! f(1)$, so taking $s=0$ we get $f(1)=1$. Also we have that
$f$ is asymptotically real
$$
\lim_{x\in \RR, x\to +\infty} \Im f(x) =0
$$
and this can be used in place of the condition $f'(1) \in \RR$ (see the Remarks after Theorem \ref{thm:main-Euler-gamma}).
Also, if $\rho \in \CC_+$ is a zero or pole of $f$, the $\rho +n$ is also a zero or pole for an integer $n\geq 0$ and the asymptotic 
would not hold, hence the divisor of $f$ in contained in $\CC-\CC_+$. In the half plane $\CC_+$ the Weierstrass asymptotic implies 
that $f$ satisfies an order $1$ estimate, so $f$ is a meromorphic function of finite order by Proposition \ref{prop:domain}. 
Using Theorem \ref{thm:main-Euler-gamma},
we have $f=\Gamma$.

\medskip
$\bullet$ \textbf{Binet-Malmst\'en and Gauss integral formulas.}
\medskip

Now we derive Binet-Malmst\'en integral formula for $\log \Gamma (s)$ (1849, \cite{Mal1849}). The formula is usually 
attributed to Malmst\'en but we can previously find it in Binet (1939, \cite{Bi}) and N\"orlund  attributes the 
formula to Plana.

\begin{theorem}[Binet-Malmst\'en] \label{thm:Malmsten}
We have for $\Re s >-1$
$$
\log \Gamma (s+1) = \int_0^{+\infty} \left ( s- \frac{1-e^{-st}}{1-e^{-t}}\right ) \frac{e^{-t}}{t} \, dt .
$$
\end{theorem}

\begin{proof}
We define for $s\in \CC_+$,
$$
f(s+1) =\exp \left ( \int_0^{+\infty} \left ( s - \frac{1-e^{-st}}{1-e^{-t}}\right ) \frac{e^{-t}}{t} \, dt\right )
$$
Then, since the integral is finite for all $s\in \CC_+$, it is clear that $f$ is a holomorphic function with no zeros in $\CC_+$. Making $s=0$ 
gives $f(1)=e^0=1$. We have
$$
\left | s - \frac{1-e^{-st}}{1-e^{-t}}\right | \leq Ct |s|^2
$$
thus the integral is $\cO(|s|^2)$ and $f$ is of finite order. Also $f$ is real analytic. 
We have,
$$
\log f(s+1)-\log f(s) = \int_0^{+\infty} (1- e^{-(s-1)t})  \frac{e^{-t}}{t} \, dt =\log s
$$
where the last equality is the Frullani integral: Starting from
$$
\frac1s =\int_0^{+\infty} e^{-st} \, dt
$$
and integrating in the $s$ variable between $1$ and $s$, we get
$$
\log s = \int_0^{+\infty} \frac{e^{-t} - e^{-st}}{t} \, dt.
$$
Therefore, using Theorem \ref{thm:main-Euler-gamma} we get $f(s)=\Gamma(s)$.
\end{proof}

Taking the derivative of Malmst\'en fomula we get Gauss integral formula for the logarithmic derivative (1812, \cite{Ga}).

\begin{theorem}[Gauss formula] \label{thm:Gauss_formula}
We have for $\Re s >-1$
  \begin{equation*}\label{eq:Gauss_for_Gamma}
    \frac{\Gamma'(s+1)}{\Gamma (s+1)}=\int_0^{+\infty} \left (1-\frac{t e^{-st}}{1-e^{-t}} \right ) \frac{e^{-t}}{t}\, \dd t 
     \end{equation*}
\end{theorem}

\medskip
$\bullet$ \textbf{Definition \`a la Lerch.}
\medskip

For $\Re t>1$, $s\in \CC-\RR_-$ and $k\in \NN$, we consider the branch of $(s+k)^{-t}$ given 
by the principal branch of the logarithm,
$$
(s+k)^{-t}=e^{-t\log (s+k)} \ .
$$
We consider the Hurwitz zeta function defined for $\Re t>1$ and $s\in \CC-\RR_-$ by
\begin{equation}\label{eqn:Hurwitz_zeta}
\zeta(t,s)=\sum_{k=0}^{+\infty} \frac{1}{(s+k)^t} \ .
\end{equation}
For $s=1$, $\zeta(t)=\zeta(t,1)$ is the Riemann zeta function. 
We derive Lerch formula:

\begin{proposition}[Lerch formula]
Given $s\in \CC$, the Hurwitz zeta function has a meromorphic continuation to $t=0$ and 
\begin{equation}
\log \Gamma(s) = \left [\frac{\partial}{\partial t}\zeta(t,s) \right ]_{t=0}-\zeta'(0) 
\end{equation}
\end{proposition}

The analytic extension to the half plane $\Re t >-1$ follows from the following Lemma.
\begin{lemma}
We have for $s\in \CC$ and $\Re t> -1$
$$
\zeta(t,s) =\frac{s^{1-t}}{t-1} + \frac{s^{-t}}{2} -t \int_0^{+\infty} \frac{u-[u]-1/2}{(u+s)^{t+1}} \, du
$$
and 
$$
\left [\frac{\partial}{\partial t}\zeta(t,s) \right ]_{t=0} = \left (s-\frac12 \right ) \log s -s -\int_0^{+\infty} \frac{u-[u]-1/2}{u+s} \, du
$$
\end{lemma}

\begin{proof} This follows from Euler-McLaurin formula, but more directly 
we split the integral and integrate by parts, for $\Re t>1$ and $s\in \CC-\RR_-$,
\begin{align*}
\int_0^{+\infty} \frac{u-[u]-1/2}{(u+s)^{t+1}} \, du &=\sum_{k=0}^{+\infty}  \int_k^{k+1} \frac{u-[u]-1/2}{(u+s)^{t+1}} \, du =\sum_{k=0}^{+\infty}  \int_k^{k+1} \frac{u-k-1/2}{(u+s)^{t+1}} \, du \\
&=\sum_{k=0}^{+\infty}  \left ( \left [ -t^{-1} \frac{u-k-1/2}{u+s} \right ]_k^{k+1} + t^{-1} \int_k^{k+1} \frac{du}{(u+s)^{t}}  \right ) \\
&=-t^{-1} \sum_{k=0}^{+\infty}  \frac12 \left (\frac{1}{(k+1+s)^t}  +  \frac{1}{(k+s)^t} \right ) +\frac{1}{t-1}  \left [ \frac{1}{(u+s)^{t-1}}\right ]_k^{k+1}  \\
&=-t^{-1} \left (\zeta(t,s)-\frac{s^{-t}}{2} - \frac{s^{1-t}}{t-1} \right )  \\
\end{align*}
and we get the first expression. The integral in the right hand side is holomorphic for $\Re t> -1$, and the right hand side is meromorphic for $\Re t> -1$, hence $\zeta(t,s)$ has 
a meromorphic extension for $\Re t> -1$. Taking the partial derivative of this formula and making $t=0$ we get the second expression.
\end{proof}

\begin{proof}
Now, consider 
$$
g(s) = \left [\frac{\partial}{\partial t}\zeta(t,s) \right ]_{t=0}
$$
From the expression computed in the Lemma, the integral being $\cO(s^{-1})$, the function $g$ is holomorphic on $\CC_+$ and $g(s)=\cO(|s| \log |s|)$. Also it is clearly 
real analytic since $\zeta(t,s)=\overline{\zeta(t, \bar s)}$.
Thus the function
$$
f(s)=e^{g(s)-g(0)}
$$
is holomorphic, without zeros nor poles on $\CC_+$, satisfies a finite order estimate in $\CC_+$, $f(1)=1$, and real analytic. 

Also we have 
$$
\zeta(t,s+1) -\zeta(t,s) = s^{-1}
$$
and differentiating we get the functional equation for $f$, $f(s+1) =s f(s)$.  The application of Theorem \ref{thm:main-Euler-gamma} proves 
that $f(s)=\Gamma (s)$. Finally, from the formulas we computed for the partial derivative we check that 
$$
\zeta'(0) =\left [\left [\frac{\partial}{\partial t}\zeta(t,s) \right ]_{t=0}\right ]_{s=1}
$$
\end{proof}

\section{Appendix: Basic entire function theory.}

We refer to \cite{Boa} for the following classical results.

Let $f:\CC \to \CC$ be an entire function. The order of $f$, $0\leq \rho(f) \leq +\infty$, is 
$$
\rho (f)=\limsup_{s\to \infty} \frac{\log \log |f(s)|}{\log |s|}
$$
(see \cite{Boa} Chapter 2 p.8)
Hence, an entire function of finite order satisfies a \textit{finite order estimate}, for some $C,d>0$,
for any $s\in \CC$
\begin{equation}\label{eq:finite_order_estimate}
|f(s)|\leq C e^{|s|^d}
\end{equation}
and, conversely, any entire function satisfying some finite order estimate is of finite order. For a domain 
$\Omega\subset \CC$, we say that an holomorphic function $f$ on $\Omega$ satisfies a finite order estimate
in $\Omega$ if it satisfies (\ref{eq:finite_order_estimate}) for $s\in \Omega$.

\medskip

A function of order $\rho=1$ is of exponential type $0\leq \tau (f)\leq +\infty$ if 
$$
\tau (f) = \limsup_{s\to \infty} \frac{\log |f(s)|}{|s|}
$$
(see \cite{Boa} Chapter 2 p.8)
A function of order $1$ has finite exponential type if it satisfies a \textit{finite exponential type estimate}, 
for some $C, \mu >0$, for any $s\in \CC$,
\begin{equation}\label{eq:finite_type_estimate}
|f(s)|\leq C e^{\mu |s|}
\end{equation}
and, conversely, any entire function satisfying some finite exponential type estimate is of finite exponential type. 
For a domain 
$\Omega\subset \CC$, we say that an holomorphic function $f$ on $\Omega$ satisfies a finite exponential type estimate
in $\Omega$ if it satisfies (\ref{eq:finite_type_estimate}) for $s\in \Omega$.

\medskip

A fundamental result by Weierstrass is that any meromorphic function in $\CC$ is the quotient of two entire functions 
in $\CC$. A meromorphic function in $\CC$ is of finite order if it is the quotient of two entire functions of 
finite order. The space of meromorphic functions of finite order is a field. More precisely, the space of meromorphic 
functions of order $\leq d$ is a field. 

\medskip

The \textit{convergence exponent} $0\leq \rho_1\leq +\infty$, of a positive 
divisor $D=\sum_\rho n_\rho .\rho$ is the infinimum of the exponents $\alpha >0$ such that 
$$
\sum_{\rho \not= 0} \frac{n_\rho}{|\rho|} <+\infty
$$
(see \cite{Boa} Definition 2.5.2 p.14)
The smallest positive integer $\alpha >0$ for which we have convergence is denoted by $p+1$ and $p$ is the 
genus of the divisor (see \cite{Boa} Definition 2.5.4 p.14).
The convergence exponent of an entire function $f$ is the convergence exponent of $\Div(f)$.
We have $\rho_1(f)\leq \rho(f)$ (see \cite{Boa} Theorem 2.5.18 p.17).

\medskip

If $D=\sum_\rho n_\rho .\rho$ is a divisor of finite convergence exponent $\rho_1$ and genus $p$, the canonical 
product
$$
g(s)=\prod_{\rho\not= 0} E_p(s/\rho)^{n_\rho}
$$
has order $\rho_1$ (see \cite{Boa} Theorem 2.6.5 p.19).

\medskip

\textbf{Acknowledgements.} I thank you Professors A. Aycock and Ch. Berg for feedback and corrections.

\end{document}